\documentclass[letterpaper, 10 pt]{article}  

\usepackage{graphics} 
\usepackage{epsfig} 
\usepackage{mathptmx} 
\usepackage{times} 

\usepackage{amsmath,amssymb,amsthm}
\usepackage[utf8]{inputenc}
\usepackage{microtype}
\usepackage{todonotes}
\usepackage{cancel}

\usepackage{url, hyperref}

\def\fkernelmat{\mathcal{K}^\alpha}
\def\z{\mathsf{z}}
\def\F{\mathsf{F}}
\def\L{\mathcal{L}}
\def\D{\mathcal{D}}
\def\a{\mathfrak{a}}
\def\b{\mathfrak{b}}
\def\Permut{P}

\def\R{\mathbb{R}}
\def\Z{\mathbb{Z}}
\def\P{\mathbb{P}}

\DeclareMathOperator{\blkdiag}{block\ diagonal}
\DeclareMathOperator{\stack}{vec}
\DeclareMathOperator{\sinc}{sinc}
\newcommand{\argmin}[1]{\underset{#1}{\operatorname{arg}\,\operatorname{min}}\;}

\newtheorem{theorem}{Theorem}

\newtheorem{proposition}[theorem]{Proposition}

\title{\LARGE \bf
Machine Learning Barycenter Approach to Identifying LPV State--Space Models
}

\author{Rodrigo~A.~Romano, 
	P. Lopes dos Santos, 
	Felipe~Pait, \\
	T-P Perdico\'{u}lis, 
	and Jos\'{e} A. Ramos 
}

\begin{document}

\maketitle
\thispagestyle{empty}
\pagestyle{empty}

\begin{abstract}

In this paper an identification method for state--space LPV models is presented. The method is based on a particular parameterization that can be written in linear regression form and enables  model estimation to be handled using Least--Squares Support Vector Machine (LS--SVM). The regression form has a set of design variables that act as filter poles to the underlying basis functions. In order to preserve the meaning of the Kernel functions (crucial in the LS--SVM context), these are filtered by a 2D--system with the predictor dynamics. A data--driven, direct optimization based approach for tuning this filter is  proposed. 
The method is assessed using a simulated example and the results obtained are twofold. First, in spite of the difficult nonlinearities involved, the nonparametric algorithm was able to learn the underlying dependencies on the scheduling signal.  Second,  a significant improvement in the performance of the proposed method is registered, if compared with the one achieved by placing the predictor poles at the origin of the complex plane, which is equivalent to considering an estimator based on an LPV auto--regressive 
structure.

\end{abstract}

\section{LPV STATE--SPACE MODEL IDENTIFICATION} \label{sec:lpvsysid}

Linear parameter--varying (LPV) models have proved to be particularly suitable for describing nonlinear and time--varying systems.  As a consequence, the development of efficient algorithms to estimate their parameters has attracted considerable research efforts \cite{Toth:2010a, LopesDosSantos:2012}. 

In most existing approaches, the system parameters  specify  linear combinations of basis functions, whose selection is a critical issue, because it is often necessary to use a large set of basis functions in order to capture unknown dependencies.  Usually  this leads to overparameterized model structures with sparse true parameter vectors \cite{Novara:2012, Toth:2009, Toth:2012b}.  In addition, if there is no prior information available, the chosen basis functions may be inadequate, leading to potential structural bias. 

Machine learning approaches based on the Least--Squares Support Vector Machine (LS--SVM) framework  that cope with the above mentioned issues have emerged recently \cite{Toth:2011,Laurain:2012,Bachnas:2014}. 
 Although the mainstream of parameter-varying control design methods employs state--space representations, most of the machine learning approaches found in the literature refer to input--output model structures.  Namely, a nonparametric approach employing state--space model structures and relying on LS--SVM for estimation of the output vector has been proposed by \cite{Santos:2014}, and  \cite{Rizvi:2015} describes a full nonparametric algorithm but under the assumption of measurable states, while \cite{Proimadis:2015} proposes a subspace method, which is known to be suitable for low dimensional problems only.

This paper presents an algorithm that uses the LS--SVM framework to identify state--space LPV models. The parameterization considered admits an output predictor which is linear in the model parameters and possesses a set of user--defined poles, which can be chosen to filter the noise from the data. A data--driven approach based on direct (derivative--free) optimization is proposed for  tuning this filter. Due to the linearity of the predictor, the parameter estimation problem is formulated in a LS--SVM setting. 
This feature makes it possible to depict a wide range of nonlinear dependencies as linear combinations of infinitely many functions defined through the choice of a particular inner product and a relative low dimension parameter vector.  However, the  model structure considered requires the basis functions to be filtered by the predictor before taking part in the inner product.  To preserve the meaning of the Kernel functions (crucial in the LS--SVM context), it is proposed to filter the kernel matrix by a  2D--system with the predictor dynamics.

The paper has the following structure: The motivation for the study of this problem as well as a brief review of the state--of--art in the field are presented in Section~\ref{sec:lpvsysid}.  In Section~\ref{sec:lpvpar}, the system parameterization is shown and the linear predictor derived.  The formulation of the LPV model as an LS--SVM is described in Section~\ref{subsec:flssvm}, and the data--driven approach for tuning the predictor poles is explained in Section~\ref{subsec:ddfiltertun}.  In Section~\ref{sec:simulations}, a set of Monte Carlo runs  on a  simulated example is carried out and the performance of the algorithm is compared with that of a standard auto--regressive (LPV--ARX) approach.  In Section~\ref{sec:conclusion}  conclusions are drawn
 and  directions for future work are outlined.

\section{LPV MODEL PARAMETERIZATION} \label{sec:lpvpar}

The  discrete--time linear parameter varying (LPV) systems considered are of the form
\begin{align}
x_{k+1} & = \left( A + L(p_k)C \right) x_k + B(p_k)u_k  \label{eq:moli_state}\\
y_k & = C x_k , \label{eq:moli_out}
\end{align}
where $x_k \in \R^{n_x}$, $u_k \in \R$ and $y_k \in \R$ are the state, input and output, respectively. 
 The constant matrices $C$ and $A$ can be freely chosen, provided $A$ is stable and the pair $(C,A)$ is observable. The time--varying parameter vectors $L(p_k) \in \R^{n_x}$ and $B(p_k) \in \R^{n_x}$ are given by
\begin{align}
L(p_k) & = \sum_{r=1}^{n_f} L_r f_r (p_k), \label{eq:Lpk} \\
B(p_k) & = \sum_{r=1}^{n_f} B_r f_r (p_k), \label{eq:Bpk}
\end{align}
 where the scheduling signal $p_k: \Z \to \P$ is assumed to be known in each sampling instant $k$. The set $\P \subseteq \R^{n_p}$ denotes the scheduling space and $f_r(p_k) : \P \to \R$ are arbitrary basis functions whose contributions are weighted by $L_r$ and $B_r$, $r \in \{1, \ldots, n_f \}$.

This state--space parameterization is inspired by an LTI linear time--invariant structure proposed in the context of adaptive control theory \cite{Pait:1991}. Due to the observability of the pair $(C,A)$,  for a constant scheduling signal (that is, $p_k = \bar{p}$ for all $k \in \Z$), the structure \eqref{eq:moli_state}--\eqref{eq:moli_out} matches any transfer function of McMillan degree not exceeding $n_x$. Define the parameter vector
\begin{equation}
\theta = \stack \left( \begin{bmatrix}
L_1 \;\; \cdots \;\; L_{n_f} \;\; 
B_1 \;\; \cdots \;\; B_{n_f}
\end{bmatrix} \right) \in \R^{2 n_x n_f},
\end{equation}
where the operator $\stack \left( \cdot \right)$ stacks the columns of the argument on top of each other. As shown next, it is possible to construct a predictor to $y_k$, which is linear with respect to $\theta$. 
\begin{proposition} \label{prop:predictor}
The LPV model \eqref{eq:moli_state}--\eqref{eq:moli_out} admits a predictor described by the realization
\begin{align}
\varphi_{k+1} & = \mathcal{A} \varphi_k +
\mathcal{B} \left( \begin{bmatrix} y_k \\ u_k \end{bmatrix} \otimes \F(p_k) \right)  \label{eq:scalar_pred_state} \\
\hat{y}_k & = \theta^\top \varphi_k \label{eq:scalar_pred_eqn}
\end{align}
where
\begin{align*}
\mathcal{A} & \triangleq \blkdiag\{A^\top, \ldots, A^\top \} \in \R^{2 n_x n_f \times 2 n_x n_f}  \\
\mathcal{B} & \triangleq \blkdiag\{C^\top, \ldots, C^\top \} \in \R^{2 n_x n_f \times 2 n_f} \\
\F(p_k) & \triangleq \begin{bmatrix} f_1(p_k) & \cdots & f_{n_f}(p_k) \end{bmatrix}^\top \in \R^{n_f} .
\end{align*}
\end{proposition}
\begin{proof}
State equation \eqref{eq:moli_state} can  be written as
\begin{equation} \label{eq:compact_moli_state}
x_{k+1} = A x_k + \mathbf{L} \F(p_k)y_k + \mathbf{B} \F(p_k)u_k .
\end{equation}
with 
\begin{align*}
\mathbf{L} & \triangleq \begin{bmatrix} L_1 & \cdots & L_{n_f} \end{bmatrix} \in \R^{n_x \times n_f} \\
\mathbf{B} & \triangleq \begin{bmatrix} B_1 & \cdots & B_{n_f} \end{bmatrix} \in \R^{n_x \times n_f},
\end{align*}
and $\F(p_k)$ as defined above. From \eqref{eq:moli_out} and \eqref{eq:compact_moli_state} we have
\begin{equation} \label{eq:prop1_output_eqn}
\hat{y}_k = C(qI_{n_x}-A)^{-1} \left( \mathbf{L} \F(p_k)y_k + \mathbf{B} \F(p_k)u_k \right),
\end{equation}
where $I_{n_x}$ is an identity matrix of dimension $n_x$ and $q$ is the forward shift operator, i.e., $qu_k = u_{k+1}$. The notation $\hat{y}_k$ is used to emphasize that the previous expression aims to predict $y_k$ based on  \eqref{eq:moli_state}--\eqref{eq:moli_out}, given input--output samples up to instant $k-1$. The first term in the right hand side of \eqref{eq:prop1_output_eqn} can be written as
\begin{equation*}
C(qI-A)^{-1} \mathbf{L} \F(p_k)y_k
= \sum_{l=0}^\infty C A^l q^{-(l+1)} \mathbf{L} \F(p_k)y_{k} .
\end{equation*}
Applying the $\stack \left( \cdot \right)$ operator to the previous equation yields
\begin{align}
\stack &\left(C(qI_{n_x}-A)^{-1} \mathbf{L} \F(p_k)y_{k}\right)  \label{eq:aug_obsv_Ly} \\
 & = \sum_{l=0}^\infty q^{-(1+l)} \left( \F^\top(p_k)y_{k} \otimes C A^l \right) \stack \left( \mathbf{L} \right) \nonumber \\
 & = \sum_{l=0}^\infty  q^{-(1+l)} \left( \F^\top(p_k)y_{k} \otimes C \right) \left( I_{n_f} \otimes A^l \right) \stack \left( \mathbf{L} \right) \nonumber \\
 & = \stack \left( \mathbf{L} \right)^\top q^{-1} \sum_{l=0}^\infty q^{-l} \left( I_{n_f} \otimes \left(A^l\right)^\top \right) \left( \F(p_k)y_{k} \otimes C^\top \right) \nonumber \\
 & = \stack \left( \mathbf{L} \right)^\top \left( qI_{n_x n_f} - I_{n_f} \otimes A^\top \right)^{-1} \left( \F(p_k)y_{k} \otimes C^\top \right). \nonumber 
\end{align}
Analogously, it follows that
\begin{align}
\stack &\left(C(qI_{n_x}-A)^{-1} \mathbf{B} \F(p_k)u_{k}\right) = \nonumber \\
& \stack \left( \mathbf{B} \right)^\top \left( qI_{n_x n_f} - I_{n_f} \otimes A^\top \right)^{-1} \left( \F(p_k)u_{k} \otimes C^\top \right). \label{eq:aug_obsv_Bu}
\end{align} 
Substituting \eqref{eq:aug_obsv_Ly} and \eqref{eq:aug_obsv_Bu} into \eqref{eq:prop1_output_eqn} results
\begin{align}
\hat{y}_k & = \begin{bmatrix}
\stack \left( \mathbf{L} \right)^\top \; \stack \left( \mathbf{B} \right)^\top
\end{bmatrix} \nonumber \\
& \cdot \left( qI_{2 n_x n_f} - I_{2 n_f} \otimes A^\top \right)^{-1}
\left( \begin{bmatrix}
\F(p_k)y_{k} \\ \F(p_k)u_{k}
\end{bmatrix} \otimes C^\top \right) \nonumber \\
& = \theta^\top \underbrace{\left(q I_{2 n_x n_f} - \mathcal{A}\right)^{-1} \underbrace{ \left( I_{2 n_f} \otimes C^\top \right) }_{=\mathcal{B}} \begin{bmatrix}
\F(p_k)y_{k} \\ \F(p_k)u_{k}
\end{bmatrix}}_{\varphi_k} \nonumber ,
\end{align}
which equals the input--output description of \eqref{eq:scalar_pred_state}--\eqref{eq:scalar_pred_eqn}.
\end{proof}

Notice that the user defined matrix $A$ determines the dynamics of the realization \eqref{eq:scalar_pred_state}--\eqref{eq:scalar_pred_eqn}. Thus, the eigenvalues of $A$ can be seen as design variables (or hyper--parameters) able to filter out the noise from the data. A data--driven, derivative--free approach to adjust these design variables is presented in Section~\ref{subsec:ddfiltertun}.  
If $A$ is in companion form, then an observable pair $(C,A)$ is constructed by choosing $C$ as a matrix filled with $0$s, except for one entry, which is set to $1$. Such choice is particularly 
 convenient because leads to a $\varphi_k$ composed of delayed versions of $f_r(p_k)y_k$ and $f_r(p_k)u_k$, for $r \in \{1,\ldots,n_f\}$. For example, the pair 
\begin{align} \label{eq:companion_A}
A &= \begin{bmatrix}
0 & \cdots & 0 & -\alpha_{n_x} \\
1 & &	& -\alpha_{n_x - 1} \\
 & \ddots & & \vdots \\
 & & 1 & -\alpha_{1} 
\end{bmatrix} \\
\label{eq:can_C}
C &= \begin{bmatrix}
0 & \cdots & 0 & 1  
\end{bmatrix}
\end{align}
has a regressor $\varphi_k$ given by
\begin{align}
\varphi_k & = \left(q I_{2 n_x n_f} - \mathcal{A}\right)^{-1} \mathcal{B} \left( \begin{bmatrix}
y_{k} \\ u_{k}
\end{bmatrix} \otimes \F(p_k) \right)
\nonumber \\
& = \frac{q^{n_x}}{\alpha(q)} \left(I_{2n_f} \otimes \begin{bmatrix} q^{-n_x} & \cdots & q^{-1} \end{bmatrix}^\top \right) \left( \begin{bmatrix}
y_{k} \\ u_{k}
\end{bmatrix} \otimes \F(p_k) \right), \nonumber 
\end{align}
which can be rewritten as
\begin{align}
\varphi_k & = \frac{q^{n_x}}{\alpha(q)}
\Big[ f_1(p_{k-n_x})y_{k-n_x} \; \cdots \; f_1(p_{k-1})y_{k-1} \; \cdots \nonumber\\
& \qquad \qquad f_{n_f}(p_{k-n_x})y_{k-n_x} \; \cdots \; f_{n_f}(p_{k-1})y_{k-1} \nonumber \\[5pt]
& \qquad \qquad f_1(p_{k-n_x})u_{k-n_x} \; \cdots \; f_1(p_{k-1})u_{k-1} \; \cdots \nonumber \\
& \qquad \qquad f_{n_f}(p_{k-n_x})u_{k-n_x} \; \cdots \; f_{n_f}(p_{k-1})u_{k-1} \Big]^\top ,  \label{eq:simple_info_vector}
\end{align}
where 
  \begin{equation}\label{eq:alpha}
  \alpha(q) = q^{n_x}\left(1+\alpha_{1}q^{-1}+\ldots +\alpha_{n_x}q^{-n_x}\right)
  \end{equation}
   is the characteristic polynomial of $A$.
  
  Hereafter, the user defined pair $(C,A)$ is assumed to be in the form \eqref{eq:companion_A}--\eqref{eq:can_C} to take advantage of this simplified form of $\varphi_k$. For the sake of compactness, it is also possible to consider a rearranged version of \eqref{eq:simple_info_vector} generated using a permutation matrix $\Permut$ such  that
\begin{equation} \label{eq:Pvarphi_k}
\Permut \varphi_k  = \frac{q^{n_x}}{\alpha(q)} \phi_k  ,
\end{equation}
with
\begin{equation}
\label{TP_phik}
\phi_k \triangleq \big[  \F^\top(p_{k-n_x}) \otimes \z_{k-n_x}^\top \;\; \cdots \;\; \F^\top(p_{k-1}) \otimes \z_{k-1}^\top \big]^\top
\end{equation}
and $\z_k \triangleq \left[y_k \;\; u_k \right]^\top$.
  Then, for a dataset of $N$ input--output samples, the regression model \eqref{eq:scalar_pred_eqn} can be written as
\begin{equation} \label{eq:matricial_regressor_eqn}
Y = \frac{q^{n_x}}{\alpha(q)} \Phi P \theta,
\end{equation}
where
\begin{align}
Y & = \begin{bmatrix} y_{n_x+1} & \cdots & y_N \end{bmatrix}^\top \nonumber \\
\label{TP_Phi}
\Phi & = \begin{bmatrix} \phi_{n_x+1} & \cdots & \phi_N \end{bmatrix}^\top.
\end{align}

\section{MODEL ESTIMATION ALGORITHM}

 Parameter estimation can be carried out by explicitly defining the basis functions $\F(p_k)$. 
Then the 
 parameter vector (composed 
  of the basis functions coefficients) is estimated by minimizing the criterion $\|Y - \Phi \theta\|_2^2$. But the success of this approach relies on an adequate choice of the basis functions, which often requires a complicated analysis of a first--principle model \cite{Bachnas:2014}. The least--squares support vector machines (LS--SVM) provides an efficient alternative to circumvent  the challenging $p$--dependent basis functions selection problem \cite{Laurain:2012, Toth:2011}. In what follows, a LS--SVM based method to estimate a nonparametric model to \eqref{eq:moli_state}--\eqref{eq:moli_out} is presented.

\subsection{Least--squares support vector machines (LS--SVM) solution } \label{subsec:flssvm}

From \eqref{eq:matricial_regressor_eqn}, the parameter estimation problem can be formulated as
\begin{equation} \label{eq:par_est_problem}
\hat{\theta} = \argmin{\theta} \frac{1}{2} \theta^\top \theta + \frac{\gamma}{2} E^\top E ,
\end{equation}
such that $$E =  Y - \frac{q^{n_x}}{\alpha(q)} \Phi \Permut \theta .$$
The regularization term $\gamma \in \R_+^\ast$ is introduced to adjust the bias--variance trade--off. In order to solve \eqref{eq:par_est_problem}, the Lagrangian
\begin{equation*} 
\L(E,\theta,\lambda) = \frac{1}{2} \theta^\top \theta + \frac{\gamma}{2} E^\top E - \lambda^\top\left( E - Y + \frac{q^{n_x}}{\alpha(q)} \Phi \Permut \theta \right)
\end{equation*}
is introduced, where $\lambda \in \R^{N-n_x}$ is the vector of the Lagrange multipliers. The Karush--Kuhn--Tucker (KKT) conditions for $\L(E,\theta,\lambda)$ 
 are given by
\begin{align}
\frac{\partial\L}{\partial E} = 0 \Rightarrow & 
E = \frac{\lambda}{\gamma} \label{eq:KTcon1}\\
\frac{\partial\L}{\partial \theta} = 0 \Rightarrow & 
\theta = \Permut^\top \Phi^\top \frac{q^{n_x}_{r}}{\alpha(q_{r})} \lambda\label{eq:KTcon2}\\
\frac{\partial\L}{\partial \lambda} = 0 \Rightarrow & 
E = Y - \frac{q^{n_x}_{l}}{\alpha(q_{l})} \Phi \Permut \theta . \label{eq:KTcon3}
\end{align}
The notation $q_l$ and $q_r$ is introduced in place of $q$ ($l$ stands for left and $r$ for right), to distinguish between the forward shift operators for the elements of $\Phi$ and $\Phi^\top$, respectively. Substituting \eqref{eq:KTcon1}--\eqref{eq:KTcon2} into \eqref{eq:KTcon3}, the vector $\lambda$ is obtained by solving
\begin{equation} \label{eq:lagrange_mult}
\left( \frac{I_{(N-n_x)}}{\gamma} + \frac{q_{l}^{n_x}}{\alpha(q_{l})} \Phi\Phi^\top \frac{q_{r}^{n_x}}{\alpha(q_{r})} \right) \lambda = Y .
\end{equation}

Firstly, suppose that the eigenvalues of $A$ are set to zero, which implies $\alpha(q_{l}) = q_{l}^{n_x}$ and $\alpha(q_{r}) = q_{r}^{n_x}$. In this particular case, \eqref{eq:lagrange_mult} simplifies to $(\gamma^{-1} I_{(N-n_x)} + \Phi\Phi^\top)\lambda = Y$. Therefore, in order to solve $\lambda$, it is enough to compute $\Phi\Phi^\top$
, whose $(i,j)$th entry is
given by 
\begin{align}
\left[\Phi\Phi^\top\right]_{ij} & = \phi_{i+n_x}^{\top}\phi_{j+n_x} 
 = \left( \F^\top(p_{i}) \otimes \z_{i}^\top \right)\left( \F(p_{j}) \otimes \z_{j} \right) + \ldots \nonumber
\\ & 
+ \left( \F^\top(p_{i+n_x-1}) \otimes \z_{i+n_x-1}^\top \right)\left( \F(p_{j+n_x-1}) \otimes \z_{j+n_x-1} \right) \nonumber \\
& = \sum_{m=0}^{n_x-1}  \left( \F^\top(p_{i+m}) \otimes \z_{i+m}^\top \right)\left( \F(p_{j+m}) \otimes \z_{j+m} \right) \nonumber \\
& = \sum_{m=0}^{n_x-1} \underbrace{\F^\top(p_{i+m}) \F(p_{j+m})}_{\triangleq \psi\left(p_{i+m}, p_{j+m} \right)} \z_{i+m}^\top \z_{j+m}. \label{eq:PhiPhiTentries}
\end{align}
The term $\psi\left( p_i, p_j \right)$ is a positive definite kernel function that defines the inner product $\F^\top(p_{i}) \F(p_{j})$ and is used to characterize the functions $f_r(p_k)$, for $r \in \{1,\ldots,n_f\}$. Among the many possible kernel functions (refer to \cite{Suykens:2002}), in this work we considered radial basis functions (RBF) with width $\sigma \in \R_+^\ast$. Thus, the kernel functions are calculated by the formula
\begin{equation*}
\psi\left( p_i, p_j \right) = \exp \left( - \frac{\|p_{i}-p_{j}\|_2^2}{\sigma^2}\right) ,
\end{equation*}
where $\| \cdot \|_2$ denotes the $\ell_2$--norm of the argument.

It can be shown that for the particular choice $\alpha(q)=q^{n_x}$, i.e. $\alpha(q_{l})=q_{l}^{n_x}$ and $\alpha(q_{r})=q_{r}^{n_x}$, \eqref{eq:moli_state}--\eqref{eq:moli_out} corresponds to the LPV state--space shifted--form presented in \cite{Toth:2012}, which is equivalent to the input--output (IO) representation
\begin{equation}\label{ARX-LPV}
y_k = - \sum_{m=1}^{n_\a} \a_m(p_{k-m})y_{k-m}+ \sum_{m=1}^{n_\b} \b_m(p_{k-m})u_{k-m},
\end{equation}
where $\a_m,\b_m: \P \to \R$ and $n_\a = n_\b = n_x$. Nevertheless, it still remains to address (in a LS--SVM context) the more general and interesting case, in which the roots of $\alpha(q)$ are not necessarily zero. In other words, the entries of the matrix
\begin{equation}\label{eq:fkernelmat}
\fkernelmat \triangleq \frac{q_{l}^{n_x}}{\alpha(q_{l})} \Phi\Phi^\top \frac{q_{r}^{n_x}}{\alpha(q_{r})},
\end{equation}
shall be calculated using the kernel function $\psi(p_i,p_j)$ that represent the inner--product $\F^\top(p_i)\F(p_j)$. In what follows, it is shown  that the entries of $\fkernelmat$ are the outputs of a separable--denominator 2D infinite impulse response (IIR) filter.

\begin{proposition}
Let $\alpha(q_l)$ and $\alpha(q_r)$ be defined as in \eqref{eq:alpha}. The $(i,j)$th entries of the matrix $\fkernelmat$ defined in \eqref{eq:fkernelmat} are the outputs of the 2D--system
\begin{equation} \label{eq:2Dfilter}
\fkernelmat_{i,j}=\dfrac{q_l^{n_x}q_r^{n_x}}{\alpha(q_l)\alpha(q_r)} \left[ \Phi \Phi^\top \right]_{ij}.
\end{equation}
Hence, these entries are computed by solving $\fkernelmat_{i,j}$ in
\begin{equation} \label{eq:fkernelmat_sumform}
\sum_{m_l = 0}^{n_x} \sum_{m_r = 0}^{n_x}  \alpha_{m_l} \alpha_{m_r} \fkernelmat_{i-m_l,j-m_r} = \left[ \Phi \Phi^\top \right]_{ij},
\end{equation}
considering $\alpha_0 = 1$.
\end{proposition}
\begin{proof}
From \eqref{eq:matricial_regressor_eqn} and \eqref{eq:PhiPhiTentries} it follows that the $(i,j)$th entry of \eqref{eq:fkernelmat} is given by
\begin{align*}
\fkernelmat_{i,j} & = \frac{q_{l}^{n_x}}{\alpha(q_l)} \phi_{i+nx}^\top \phi_{j+n_x} \frac{q_{r}^{n_x}}{\alpha(q_r)}\\
& =\frac{q_{l}^{n_x}}{\alpha(q_l)}  \left[\Phi\Phi^\top\right]_{ij}\frac{q_{r}^{n_x}}{\alpha(q_r)}=\dfrac{q_l^{n_x}q_r^{n_x}}{\alpha(q_l)\alpha(q_r)} \left[ \Phi \Phi^\top \right]_{ij}
\end{align*}
and this completes the proof.
\end{proof}

For $p_k=\bar{p}$, it follows from \eqref{eq:Lpk} that
\begin{align*}
L(\bar{p}) =\mathbf{L}\F(\bar{p})=\left( \F^\top(\bar{p}) \otimes I_{n_x} \right)\stack(\mathbf{L}).
\end{align*}
Using an analogous expression for \eqref{eq:Bpk}, the following equation relating the parameter varying vectors $L(\bar{p})$ and $B(\bar{p}))$ to $\theta$ is obtained
\begin{equation*}
\begin{bmatrix} L(\bar{p}) \\ B(\bar{p}) \end{bmatrix} = \left( I_2 \otimes \F^\top (\bar{p}) \otimes I_{n_x} \right)
\underbrace{\begin{bmatrix} \stack(\mathbf{L}) \\ \stack(\mathbf{B}) \end{bmatrix}}_{=\theta}.
\end{equation*}
Substituting $\theta$, from condition \eqref{eq:KTcon2}, in the previous equation and using definitions \eqref{TP_phik} and \eqref{TP_Phi} yields
\begingroup
\setlength{\arraycolsep}{4pt}
\begin{align}
\begin{bmatrix} L(\bar{p}) \\ B(\bar{p}) \end{bmatrix} & = \left( I_2 \otimes \F^\top (\bar{p}) \otimes I_{n_x} \right) \Permut^\top \Phi^\top \frac{q_{r}^{n_x}}{\alpha(q_{r})}\lambda \nonumber \\
& = \begin{bmatrix}
 \psi(\bar{p},p_1) y_1  & \cdots & \psi(\bar{p},p_{N-n_x}) y_{N-n_x} \\
\vdots & & \vdots \\
 \psi(\bar{p},p_{n_x}) y_{n_x} & \cdots & \psi(\bar{p},p_{N-1}) y_{N-1} \\
 \psi(\bar{p},p_1) u_1 & \cdots & \psi(\bar{p},p_{N-n_x}) u_{N-n_x}  \\
\vdots & & \vdots \\
\psi(\bar{p},p_{n_x}) u_{n_x} & \cdots & \psi(\bar{p},p_{N-1}) u_{N-1} 
\end{bmatrix} \frac{q_{r}^{n_x}}{\alpha(q_{r})} \lambda . \nonumber
\end{align}
\endgroup
Next, define
\begin{align*}
\psi_{k}^{y} (\bar{p}) &=\frac{q_{r}^{n_x}}{\alpha(q_{r})}\psi(\bar{p},p_k) y_k\\
\psi_{k}^{u} (\bar{p}) & =\frac{q_{r}^{n_x}}{\alpha(q_{r})} \psi(\bar{p},p_k) u_k ,
\end{align*}
 that is $\psi_{k}^{y} (\bar{p})$ and $\psi_{k}^{u} (\bar{p})$ are filtered versions of the products 
$\psi(\bar{p},p_k) y_k$ and $\psi(\bar{p},p_k) u_k$, respectively.
Thus, the values of $L(\bar{p})$ and $B(\bar{p})$ may be reconstructed through
\begin{align}
\label{Lbarp}
L(\bar{p}) & =\Psi^y(\bar{p})\lambda\\
\label{Bbarp}
B(\bar{p}) & =\Psi^u(\bar{p})\lambda ,
\end{align}
where $\Psi^y(\bar{p})$ and $\Psi^u(\bar{p})$ are the Hankel matrices
\begin{align}
\label{HankelPsiY}
\Psi^y(\bar{p})&=
\begin{bmatrix}
\psi^{y}_1(\bar{p}) & \cdots & \psi^{y}_{N-n_x}(\bar{p}) \\
\vdots & \vdots & \vdots \\
\psi^{y}_{n_x}(\bar{p}) & \cdots & \psi^{y}_{N-1}(\bar{p})
\end{bmatrix}\\
\label{HankelPsiU}
\Psi^u(\bar{p})&=
\begin{bmatrix}
\psi^{u}_1(\bar{p}) & \cdots & \psi^{u}_{N-n_x}(\bar{p}) \\
\vdots & \vdots & \vdots \\
\psi^{u}_{n_x}(\bar{p}) & \cdots & \psi^{u}_{N-1}(\bar{p})
\end{bmatrix}.
\end{align}
%

\subsection{Data--driven filter tuning} \label{subsec:ddfiltertun}

As argued before, the eigenvalues of the parameter independent matrix $A$ determines the poles of the output predictor presented in Proposition \ref{prop:predictor}. Therefore, their choice can be used to filter the disturbances and the measurement noise from the plant signals. But rather than treating all the eigenvalues of $A$ as free design parameters, the polynomial $\alpha(q)$ is  parametrized as
\begin{equation} \label{eq:butteralpha}
\alpha(q) = \prod_{m=1}^{n_x} \left( q-e^{-s_mT_s} \right) ,
\end{equation}
such that its roots are the poles of Butterworth filters with cutoff frequency $\omega_c$, namely 
\begin{equation*}
s_m = \omega_c e^{j \frac{\left(2m+n_x-1\right)}{2n_x} }
\end{equation*}
mapped to the discrete--time domain using a sampling period $T_s$.

The reasons for adopting Butterworth polynomials are twofold. First, the coefficients of $\alpha(q)$ are expressed as a function of $\omega_c$, whose numerical value has a clear physical interpretation. Second, the filter tuning becomes a search in a subspace of dimension 1. Indeed, the price to be paid for the reduction in the search space of filter parameters is less flexibility in the 2D--filter frequency response.

In this work, the filter is tuned using a derivative--free optimization method, based on index of merit $J$. To this aim, we consider a set $\Omega = \{\omega_1, \ldots, \omega_{n_\omega}\}$ of candidate values of $\omega_c$   called ``curiosity points,'' and let $J(\omega_c, \D)$ be a functional which quantifies the performance of the values $\omega_c \in \Omega$ given a data set $\D = \{y_1, u_1, p_1,\ldots,y_N, u_N,p_N\}$. Then, the filter cutoff frequency is calculated through
\begin{equation} \label{eq:barycenter}
\omega_c^\ast = \frac{\sum_{\upsilon=1}^{n_\omega} \omega_\upsilon e^{-\mu J\left(\omega_\upsilon,\D\right)}}
{\sum_{\upsilon=1}^{n_\omega} e^{-\mu J\left(\omega_\upsilon,\D\right)}}.
\end{equation}
Therefore, $\omega_c^\ast$ is the barycenter of the curiosity points $\omega_\upsilon$ weighted
by the term $e^{-\mu J\left(\omega_\upsilon,\D\right)}$. The constant $\mu \in \R_+^\ast$ is used to adjust weighting terms --- the higher $\mu$, the more $\omega_c^\ast$ tends to the element in $\Omega$ that provides the lowest $J$ (best performance). The rationale behind \eqref{eq:barycenter}
is that curiosity points which achieve better performance are given more weight than those that lead to ``worse'' results. Notice that considerable freedom is retained in the choice of $J$, as its derivatives are not required. Instead, only the numerical values of the functional has to be computed for each $\omega_\upsilon$. For this reason, the barycenter can be seen as a direct optimization method. Many alternatives to the barycenter method, which was selected for its simplicity and robustness \cite{Pait:2014}, are described in the literature \cite{Conn:2009}.

\section{CASE STUDY} \label{sec:simulations}

Consider the LPV data--generating system
\begin{align*}
x_{k+1} & = \begin{bmatrix} a_{11}(p_k) & 1 \\ a_{21}(p_k) & 0 \end{bmatrix} x_k
+ \begin{bmatrix} b_{1}(p_k) \\ b_{2}(p_k) \end{bmatrix} u_k \\
y_k & = \begin{bmatrix} 1 & 0 \end{bmatrix}  x_k + v_k,
\end{align*}
where $v_k$ is a zero--mean Gaussian white noise sequence, whose variance is adjusted in order to get an specific signal--to--noise (SNR) ratio, $\P = [-0.25, 0.25]$ and
\begin{align*}
a_{11}(p_k) & = 0.35 \sinc \left( \pi^2 p_k\right) + 1.4\\
a_{21}(p_k) & = 5p_k^2 - 0.8\\
b_{1}(p_k) & = \left\{ \begin{array}{ll}
1.5 			& \text{, for } p_k > 0.125 \\
1 + 4p_k	& \text{, for } \left| p_k \right| \leq 0.125 \\
0.5 			& \text{, for } p_k < -0.125
\end{array} \right. \\
b_{2}(p_k) & = \left\{ \begin{array}{ll}
0 			& \text{, for } p_k > 0.125 \\
0.5 - 4p_k	& \text{, for } \left| p_k \right| \leq 0.125 \\
1 			& \text{, for } p_k < -0.125 .
\end{array} \right.
\end{align*}
Such system is inspired in the so-called \textit{\r{A}str\"{o}m system} \cite{Ljung:1999}, which is extended here to the LPV framework. To investigate the performance of the presented algorithm under different noise conditions, two Monte Carlo simulations of $200$ runs were performed with different signal--to--noise ratios (SNR%
), namely $\text{SNR}=20$dB and $\text{SNR}=10$dB. In each Monte Carlo run the signal $u$ is a realization of a zero--mean white--noise binary signal with length $N=800$ samples and $p$ 
is a realization of a white--noise signal with uniform distribution in the interval $[-0.25, 0.25]$.

Besides evaluating the performance of the proposed algorithm, these experiments were also used to assess the impact of not being restricted to the choice $\alpha(q)=q^{n_x}$, particularly when the predictor poles are tuned using the barycenter formula presented in Section~\ref{subsec:ddfiltertun}. Therefore, two models were estimated in each Monte Carlo run. The first, denoted as (standard) LS--SVM, was estimated with the predictor poles set at the origin of the complex plane. 
With this predictor the algorithm 
is identical to the LS--SVM method proposed in \cite{Toth:2011} for LPV--ARX IO models. The only difference is that here the ARX--LPV IO model has a dynamic dependence on $p_k$ as shown in equation \eqref{ARX-LPV}. The second is estimated by considering $\alpha(q)$ as in \eqref{eq:butteralpha}, with $\omega_c$ given by \eqref{eq:barycenter}. This model is referred to as 2D--filter, due to the way we propose to address arbitrary predictor pole choices.

The \textit{best fit rate} (BFR) was used to evaluate the models accuracy. This index is defined as 
$$\text{BFR(\%)} = 100 \cdot \max \left(1-\frac{\|Y-\hat{Y}\|_2}{\|Y-\bar{Y}\|_2} , \,0\right) ,$$
where $\hat{Y}$ is the output simulated by the estimated model and $\bar{Y}$ is the mean of the observed output sequence $Y$. The BFR index was calculated with a noiseless input-output validation data-set. 

The index of merit
$$
J(\omega_\upsilon, \D) = \min \left(\frac{\|Y-\hat{Y}(\omega_\upsilon)\|_2}{\|Y-\bar{Y}\|_2} , \,1\right)
$$
(closely related to the BFR criterion) was used to tune the cutoff frequency of the 2D--filter via \eqref{eq:barycenter}.  The set of curiosities $\Omega$ comprised six frequencies logarithmically spread between 0.05 and 0.5 of the Nyquist frequency, i.e., $\Omega = \{0.05, 0.08, 0.13, 0.2, 0.32, 0.5\}\frac{1}{\pi}$ and $\mu$ set to $130$. The regularization parameter and the RBF kernel width were set to $\gamma = 100$ and $\sigma = 0.2$, respectively, based on a trial and error procedure. The use of the barycenter method to achieve reasonable values for these hyper--parameters will be reported in a future work.

\addtolength{\textheight}{-0.5cm}   

The histograms of the achieved BFR are presented in Figure~\ref{fig:bfr20db10dbhist}. The 2D--filter approach provided more accurate models. The results also reveal that, instead of setting $\alpha(q) = q^{n_x}$, an appropriate choice of the predictor poles leads to less variability and to more robustness.
\begin{figure}[htb]
\centering
\includegraphics[width=\linewidth]{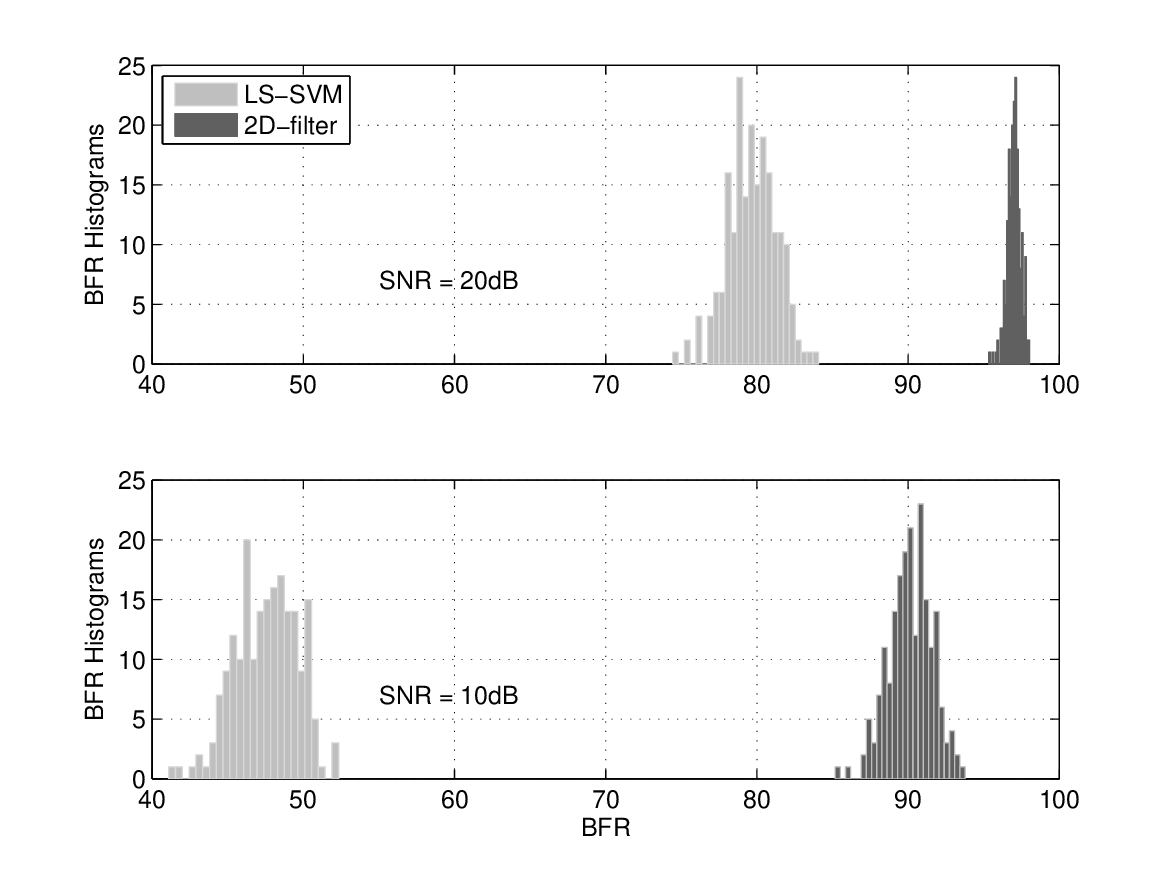}

\vspace*{-6pt}
\caption{BFR histogram achieved with each SNR.}
\label{fig:bfr20db10dbhist}
\end{figure}
These conclusions can also be drawn from Table~\ref{tab:BFRmeanstd}, that shows the mean value and the standard deviation of the BFR. The modest results provided by the standard LS--SVM method give an idea of the difficulty of the estimation problem being considered.
\begin{table}[htb]
\centering
\caption{BFR mean and standard deviation (STD).}
\label{tab:BFRmeanstd}
\begin{tabular}{ccccc}
 & \multicolumn{2}{c}{SNR=$20$dB} & \multicolumn{2}{c}{SNR=$10$dB} \\ 
\hline 
Approach & Mean & STD & Mean & STD \\ 
LS--SVM & $79.7$\% & $1.62$ & $47.6$\% & $2.05$ \\ 
2D--filter & $97.0$\% & $0.44$ & $90.2$\% & $1.43$ \\ 
\end{tabular} 
\end{table}

Figure~\ref{fig:ab_coeffs20db10db} shows that the 2D--filter significantly reduced the variance and removed the bias of the 
estimates.
\begin{figure}[!p]
\centering
\includegraphics[width=\linewidth]{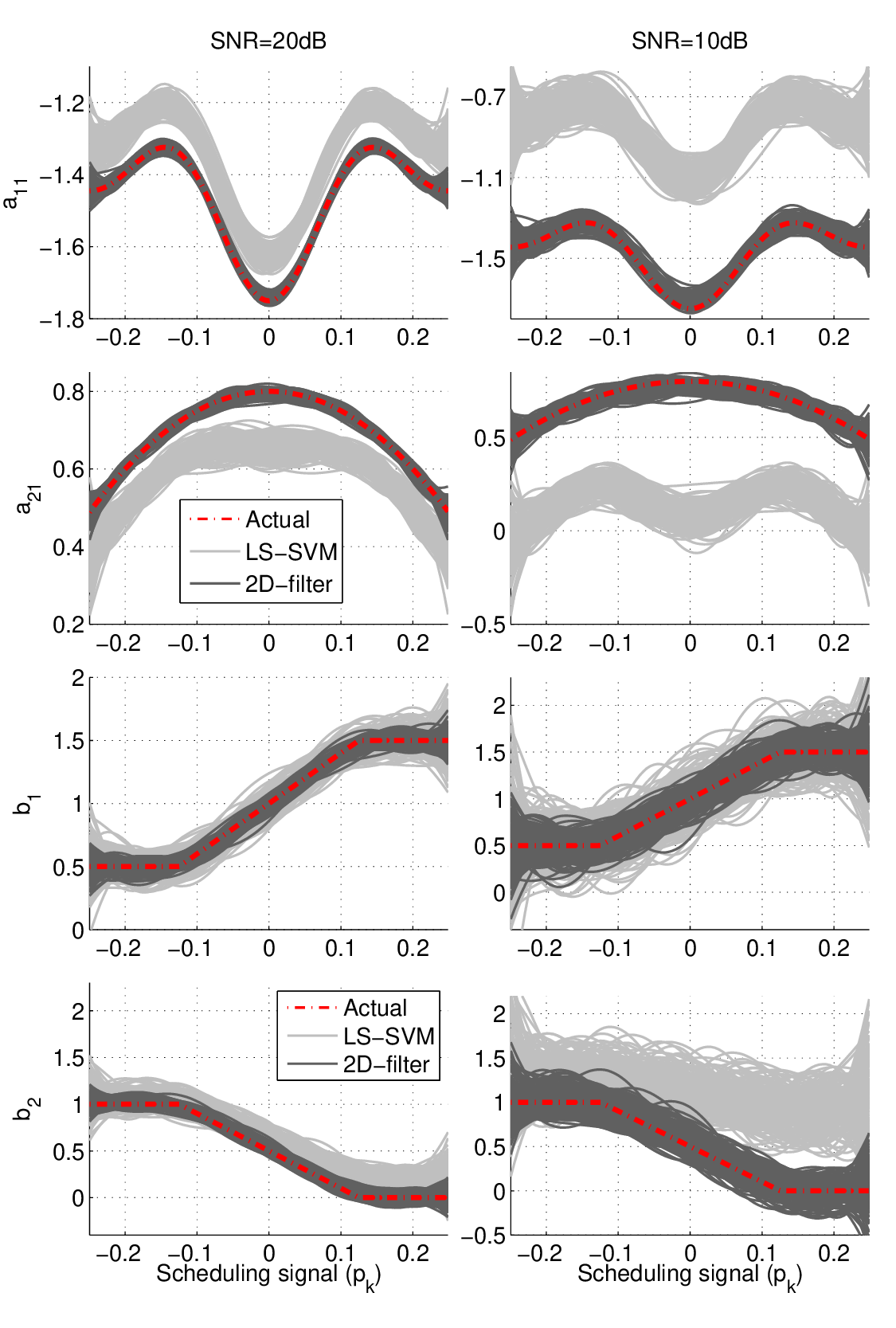}

\vspace*{-6pt}
\caption{Estimation results of the underlying coefficient functions.}
\label{fig:ab_coeffs20db10db}
\end{figure}
Indeed, the proposed method provides accurate estimates of the nonlinear $p$--dependent functions, without using any prior structural information about the system, except for the model order.

\section{CONCLUSION} \label{sec:conclusion}
This paper presents a method to estimate state--space LPV models based on the LS--SVM framework. The parameterization considered has user--defined matrices that determine a characteristic equation of the 2D system used to filter the kernel matrix. Butterworth polynomials are used to parameterize the filter characteristic equation, and a direct optimization technique (barycenter) is applied to search for its cutoff frequency. A simulated case-study showed that the proposed approach outperforms the results attained when the kernel matrix is not filtered (which is equivalent to a standard LPV--ARX estimator). This improvement is due to the flexibility in the assignment of the predictor poles handled by the 2D--filter, in combination with the data--driven tuning approach. An interesting possibility is to reformulate the proposed algorithm using the instrumental variable estimator (IV--SVM). Research concerning this topic is ongoing and shall be reported in an upcoming contribution.

\section*{ACKNOWLEDGMENT}

R Romano was supported by  Instituto Mau\'a de Tecnologia (IMT). P Lopes dos Santos and T-P Perdico\'{u}lis have been supported by FCT.


\bibliographystyle{IEEEtranS}
\bibliography{lpv_refs,molizoft_refs}

\end{document}